\documentclass[12pt]{amsart}
\usepackage{color}
\newtheorem{prop}{Proposition}[section]
\newtheorem{coro}[prop]{Corollary}
\newtheorem{thm}[prop]{Theorem}
\newtheorem{conjecture}[prop]{Conjecture}

\newcommand{\Aut}{\mathrm{Aut}}
\newcommand{\PSL}{\mathrm{PSL}} 
\newcommand{\PGL}{\mathrm{PGL}}
\newcommand{\ASL}{\mathrm{ASL}} 
\newcommand{\AGL}{\mathrm{AGL}} 
\newcommand{\AGaL}{\mathrm{A\Gamma L}} 
\newcommand{\PGaL}{\mathrm{P\Gamma L}} 
 
\newcommand{\Mat}{\mathrm{M}} 
 
\newcommand{\Cyc}{\mathrm{C}} 
\newcommand{\Alt}{\mathrm{A}} 
\newcommand{\Sym}{\mathrm{S}}

\begin{document}
\title{String C-groups as transitive subgroups of $\Sym_n$}
\author{Peter J. Cameron}
\address{
School of Mathematics and Statistics, 
University of St Andrews,
North Haugh,
St Andrews, Fife KY16 9SS,
Scotland
}
\email{pjc20@st-andrews.ac.uk}

\author{Maria Elisa Fernandes}
\address{
Center for Research and Development in Mathematics and Applications, Department of Mathematics, University of Aveiro, Portugal
}
\email{maria.elisa@ua.pt}

\author{Dimitri Leemans}
\address{Department of Mathematics, University of Auckland, Private Bag 92019, Auckland 1142, New Zealand
}
\email{d.leemans@auckland.ac.nz}

\author{Mark Mixer}
\address{Department of Applied Mathematics, Wentworth Institute of Technology, Boston, MA 02115, USA
}
\email{mixerm@wit.edu}

\date{}
\maketitle

\begin{abstract}
If $\Gamma$ is a string C-group which is isomorphic to a transitive subgroup
of the symmetric group $\Sym_n$ (other than $\Sym_n$ and the alternating group
$\Alt_n$), then the rank of $\Gamma$ is at most $n/2+1$, with finitely many
exceptions (which are classified). It is conjectured that only the
symmetric group has to be excluded.
\end{abstract}

\section{Introduction}
Classifications of string C-groups of high rank from almost simple groups has been a subject of interest for almost a decade now.
Some striking results have been obtained, for instance for the symmetric groups. Leemans and Fernandes classified string C-groups of rank $n-1$ and $n-2$ for $\Sym_n$~\cite{fl,sympolcorr} and more recently, they extended this classification to rank $n-3$ and $n-4$ with Mixer~\cite{flm}.
The key of such a classification was first to show that, from a certain value of $n$ onwards, all the maximal parabolic subgroups of a string C-group of rank $r \geq n-4$ must be intransitive. This was done by bounding the rank of string C-groups constructed from transitive groups of degree $n$ that are not $\Sym_n$ nor $\Alt_n$.
These ideas motivated the present paper where we deal with transitive groups in a more general way than in the papers cited above.
Another motivation is to contribute to the proof of the following conjecture.
\begin{conjecture}\cite[Conjecture 9.1]{flm2}
Let $n\geq 12$. The highest rank of an abstract regular polytope having $\Alt_n$ as automorphism group is $\lfloor (n-1)/2 \rfloor$.
\end{conjecture}

Suppose that $\rho_0,\ldots,\rho_{d-1}$ are involutions which
generate a string C-group $\Gamma$. This means that
\begin{itemize}
\item if $|i-j|>1$, then $\rho_i$ and $\rho_j$ commute (the \emph{string
property});
\item if $\Gamma_I$ denotes the group generated by $\{\rho_i:i\in I\}$,
for $I\subseteq\{0,\ldots,r-1\}$, then
\[\Gamma_I\cap\Gamma_J=\Gamma_{I\cap J}\]
(the \emph{intersection property}).
\end{itemize}
It is known that string C-groups are the same thing as automorphism
groups of regular polytopes~\cite[Section 2E]{ARP}. The number $d$ is the \emph{rank} of the
string C-group (or of the polytope).

Suppose that $\Gamma$ is a string C-group of rank $d$ which is a subgroup of
the symmetric group $\Sym_n$.
\begin{itemize}
\item If $\Gamma=\Sym_n$, then $d\le n-1$, with equality if and only if the
generators $\rho_i$ are the \emph{Coxeter generators} of $\Sym_n$ (that is,
$\rho_i=(i+1,i+2)$ for $0\le i\le d-1$) when $n\neq 4$, or equivalently, the polytope is
the regular simplex. Moreover, the polytopes (or string C-groups) of
ranks $n-2$, $n-3$ and $n-4$ with $\Gamma=\Sym_n$ have been classified~\cite{fl,flm}.
\item Examples of string C-groups of rank $\lfloor(n-1)/2\rfloor$ with
$\Gamma=\Alt_n$ have been found~\cite{flm1, flm2}. It is conjectured that this is the
largest possible rank for $n\ge12$ (see above). 
\item Ranks arbitrarily close to $n$ can be realised by string C-groups
corresponding to intransitive subgroups of $\Sym_n$. For example, if
$n=n_1+\cdots+n_k$, where $n_i>1$ for all $i$, then take Coxeter generators
for $\Sym_{n_1}$, \dots, $\Sym_{n_k}$: we obtain a string C-group 
$\Gamma=\Sym_{n_1}\times\cdots\times \Sym_{n_k}$ of rank $n-k$
by ordering the factors arbitrarily.
\end{itemize}

The groups not covered by these remarks are the transitive subgroups
of $\Sym_n$ other than $\Sym_n$ and $\Alt_n$. We prove the following theorem:

\begin{thm}\label{maintheorem}
Let $\Gamma$ be a string C-group of rank $d$ which is isomorphic to a
transitive subgroup of $\Sym_n$ other than $\Sym_n$ or $\Alt_n$. Then one of
the following holds:
\begin{enumerate}
\item $d\le n/2$;
\item $n \equiv 2 \mod{4}$, $d = n/2+1$ and $\Gamma$ is $\Cyc_2\wr \Sym_{n/2}$. The generators are
\begin{center}
$\rho_0 = (1,n/2+1)(2,n/2+2)\ldots (n/2,n)$;

$\rho_1 = (2,n/2+2)\ldots (n/2,n)$;

$\rho_i = (i-1,i)(n/2+i-1,n/2+i)$ for $2\leq i \leq n/2$.
\end{center}
Moreover the Sch\"afli type is $[2,3, \ldots, 3,4]$. 
\item $\Gamma$ is transitive imprimitive and is one of the examples appearing in Table~\ref{ploys}.
\item $\Gamma$ is primitive. In this case, $\Gamma$ is obtained from the permutation representation of degree 6 of $\Sym_5 \cong \PGL_2(5)$ and it is the 4-simplex of Sch\"afli type $[3,3,3]$.

\begin{table}
\begin{center}
\begin{tabular}{|c|c|c|c|l|}
\hline
$Degree$&$Number$&Structure&Order&Sch\"afli type\\
\hline
\hline
6&9&$\Sym_3 \times \Sym_3$&36&$[2,3,3]$\\
\hline
6&11&$2^3:\Sym_{3}$&48&$[2,3,3]$\\
6&11&$2^3:\Sym_{3}$&48&$[2,3,4]$\\
\hline
8&45&$2^4:\Sym_{3}:\Sym_{3}$&576&$[3,4,4,3]$\\
\hline
\end{tabular}
\caption{Examples of transitive imprimitive string C-groups of degree $n$ and rank $n/2+1$ for $n\leq 9$.}\label{ploys}
\end{center}
\end{table}

\end{enumerate}
\end{thm}

In Table~\ref{ploys}, Degree and Number correspond to the numbering of transitive groups in {\sc Magma}~\cite{BCP97}. For instance, the group of degree 6 and number 9 may be obtained using the function {\sf TransitiveGroup(6,9)}.
We will prove this theorem by separating two cases: the case when $\Gamma$
is transitive but imprimitive (see Proposition~\ref{transimp} and Proposition~\ref{prop2}); and the case when $\Gamma$ is primitive (see Proposition~\ref{prim}).

\section{Transitive imprimitive groups}

In this section, we assume that $\Gamma$ is a string C-group of rank $d$
which is a transitive imprimitive subgroup of $\Sym_n$.

We want to bound $d$ by a function of $n$. The cross-polytope (corresponding
to the Coxeter group $\Cyc_d$) shows that there is an example with $d=n/2$: its
automorphism group is the wreath product $\Cyc_2\wr \Sym_d$. We
will show that this is almost best possible.

Computing all transitive imprimitive string C-groups of degree $\le 9$ takes only a few seconds using {\sc Magma}. Those string C-groups of degree $n\le 9$ and rank $\ge n/2+1$ are listed in Table~\ref{ploys}.

\begin{prop}\label{transimp}
With the above hypotheses, if $n\ge10$, then $d\le n/2+1$,
with equality only when $n \equiv 2 \mod{4}$.
\end{prop}

We make some observations before getting into the proof of this proposition.

We use the fact that the centraliser (in the symmetric group) of a primitive
group is trivial, except for the cyclic groups of prime order (which we will
ignore for the moment).

Suppose that we have a set of generators for a primitive group, partitioned
into two subsets, so that the elements of the first subset commute with those
in the second. Then each subset generates a normal subgroup, and their
intersection is central, and so is trivial. Also, each normal subgroup is
transitive and semiregular, and so regular, and each is the centraliser
of the other in the symmetric group, they are minimal normal subgroups and
are products of isomorphic non-abelian simple groups (whence the degree is
at least $60$). Moreover, since the union of the two subsets generates the
group, it is the direct product of the two subgroups, so in fact each is
simple.

By Whiston's theorem~\cite{w}, a set of independent elements in $\Sym_n$ has
cardinality at most $n-1$. We also use the fact that, when $n\ge5$, a set of
$n-1$ involutions generating $\Sym_n$ and having a string diagram must be the
Coxeter generators (this follows from the classification of independent sets
of size $n-1$ in $\Sym_n$ when $n\ge 7$ see~\cite{cc}, and from {\sc Magma} computations when $n=5$ or $6$ see~\cite{LV2006}).
\begin{proof}
Suppose that $\Gamma$ is transitive but imprimitive, with $m$ blocks
of imprimitivity each of size $k$, where we have chosen the block system so
that the blocks are maximal. 
Let $L$ be a subset of the generating set which
is an independent generating set for the group induced on the set of blocks.
{\color{black} Since the blocks are maximal, the group generated by $L$ has a primitive action on the blocks.}
By the remarks above, $L$ induces a subgraph of the path with at most two
connected components. If there are two, then $L$ is an independent generating
set for $S\times S$, where $S$ is a non-abelian simple group of order $m$;
thus $|L|\le2\log_2m$ {\color{black} and in this case, $m\ge 60$}. If there is only one, then $|L|\le m-1$, with equality
only if the group induced on the blocks is the symmetric group $\Sym_m$.

Let $C$ be the set of generating involutions which commute with all the
involutions in $L$, and $R$ the remaining set of involutions. Since elements
of $R$ do not commute with elements of $L$, they are adjacent to them in the
diagram, so $|R|\le 4$, and if $L$ is connected then $|R|\le 2$.

The proof is separated into three cases.

\paragraph{\underline{Case 1:} $k>2$, $m>2$} The group induced on the blocks is not cyclic
of prime order, so its centraliser is trivial, so the elements of $C$ fix all
the blocks. Since the group they generate centralises a transitive permutation
group on the blocks, it acts in the same way on each block. Thus $|C|\le k-1$.

If $L$ has two components, then $d\le 2\log_2m+(k-1)+4$ {\color{black} and $m\ge60$. This
is then always smaller than $n/2$}. Hence we can assume that the induced subgraph on $L$ is connected, and also, that it is an
interval in the string. Since any element of $R$ is joined to an element of $L$,
$|R|\le 2$ and we have $d\le k+m$. Let us improve this by one, as follows.

Suppose that equality holds. Then the group induced on the blocks is
$\Sym_m$, and the group induced on a block by its stabiliser is $\Sym_k$; the
induced subgraph on $C$ is also connected; and the two elements of $R$
are at each end of $L$, one between $L$ and $C$, and one (say $\rho$)
at the end, and so commuting with $C$. Now the group generated by $C$
induces the symmetric group on every block, and so $\rho$ must permute
the blocks non-trivially; but then, since $\langle C\rangle$ and $\langle L \rangle$ centralise each other, and $\langle L \rangle \cong \Sym_m$, and since $\rho$ also centralises $\langle C \rangle$, $\rho$ has to be in $\langle L \rangle$, a contradiction.

Thus $d\le k+m-1\le 2+n/3$, which is at most $n/2$ as long as $n\neq 9$ which is the case here since $n\geq 10$.

\paragraph{\underline{Case 2:} $k=2$} We may again assume that the induced subgraph on $L$ is connected as in Case 1. Thus $|R| \leq 2$. As above, an involution which commutes with a group
acting primitively on the blocks must fix every block, and so must fix or
interchange the two points in each block. Let $\rho$ be such an involution.
If $\rho$ fixes the points in some but not all of the blocks, then (by the
primitivity of the group induced on the blocks) we can find an element of 
the group generated by $L$ which moves a block fixed pointwise by $\rho$ to
one not fixed pointwise, and hence not commuting with $\rho$. So $\rho$ must
interchange the points in every block. But then $\rho$ lies in the centre of
$\Cyc_2\wr \Sym_{n/2}$, the group fixing the block system. This shows that the
set $C$ {\color{black} has cardinality at most} one, and its element is the fixed-point-free involution that swaps the points inside the blocks.
Therefore the bound $d \leq k-1 + m-1 + 2 = 2 + n/2$.

Now we can show that $d\le n/2+1$ and that equality holds if and only if $n \equiv 2\mod{4}$ in this case. 
If $d=n/2+2$ then the diagram must have the $n/2-1$ nodes of $L$ with one node
of $R$ at each end {\color{black}and $C$ must be the fixed-point-free involution fixing each block.} Each element of $R$ commutes with a $\Sym_{n/2-1}$ permuting
the blocks, since the elements of $L$ correspond to standard Coxeter
generators of $\Sym_{n/2}$ as $n/2>4$. So in fact both elements of
$R$ must fix at least $n/2-1$ blocks, and so must fix every block.

The subgroup of $\Sym_n$ fixing every block is $\Cyc_2^{n/2}$. As
$\Sym_{n/2}$-module it has just four submodules: the zero module, the module
generated by the all-$1$ vector, the even-weight submodule, and the whole
module. Since the involution in $C$ corresponds to the all-$1$ vector, 
if $n\equiv 0 \mod{4}$, adding any even vector (and hence any odd vector as well) will make the all-$1$ vector dependant and break independence of the generators.
Therefore $|R| = 0$ in that case.
If $n\equiv 2\mod{4}$, adding an odd vector will break independence of the all-$1$ vector, but adding an even vector will not.
There are only two ways to add such a vector in order to satisfy the commuting property, namely by having an involution which swaps all but two points that are the two points moved by one of the generators at the end of the $L$-string. Assume we take the two. Then, as we have the symmetric group acting on the blocks, we can move the fixed pair of points of the first to the fixed pair of points of the second. So, if we take both, independence will not be preserved. Hence $|R| = 1$ in that case and the involution of $R$ swaps all but two points. Moreover, $d = n/2 + 1 < n/2 + 2$.
This case exists as will be shown in Proposition~\ref{prop2}, hence it is not possible to improve the bound here.

\paragraph{\underline{Case 3:} $m=2$} In this case, $|L|=1$, and so $|R|\le 2$. We claim that
$|C|\le k-1$ in this case too. Let $L=\{\rho\}$.
Elements of $C$ either fix or interchange the blocks. We claim that
multiplying them by powers of $\rho$ {\color{black}(i.e. by $\rho$ or the identity)} so that they fix the blocks preserves
independence. For suppose that, for $\rho_i\in C$, $\rho_i\rho^{\epsilon_i}$
fixes the blocks. Suppose that $\rho_i^{\epsilon_i}$ is a product of other
elements of this form. Since all elements of $C$ commute with $\rho$, we find
that $\rho_i$ is a corresponding product of elements $\rho_j$ times a power
of $\rho$, contradicting independence of the original set.

Thus we have $d\le1+2+(n/2-1)=n/2+2$.

Assume that equality holds.
Then the induced subgraph on $C$ is connected, and $\rho$ and the elements
of $C$ generate $\Cyc_2\times \Sym_k$. So the string
consists of $C$, followed by an element of $R$, then $\rho$, then the other
element of $R$ (say $\rho'$). Now $\rho'$ preserves the block system and
commutes with $C$, so it must necessarily be in $\langle C, L \rangle$, a contradiction.
Hence $d \le n/2+1$.

Now suppose that $d=n/2+1$. 
That means $R$ consists of one involution $\rho'$.
Assume first that $\rho'$ permutes the two blocks. In that case, we have that $\rho'\rho \in \langle C \rangle$ and therefore, $\rho' \in \langle C,L \rangle$, a contradiction.
Assume next that 
$\rho'$ does not permute the blocks. Then $\rho'^\rho \in \langle C \rangle$ and therefore, again, 
$\rho' \in \langle C,L \rangle$, a contradiction.

Hence, in this case, $d\le n/2$.
\end{proof}
Let us now classify the groups attained by the bound. These groups occur only for degree $n \equiv 2\mod{4}$ as shown in the proof above.
For $n = 6$ and $n=10$, we use {\sc Magma} to classify all string C-groups of degree $n$ coming from transitive imprimitive groups.

\begin{prop}\label{prop2}
Let $G$ be a transitive imprimitive group of degree $n \geq 11$ and let $S$ be an independent generating set of involutions satisfying the intersection property.
If $|S| = n/2 + 1$ then $n \equiv 2 \mod{4}$ and, up to isomorphism, the pair $(G,S)$ is unique.
\end{prop}

\begin{proof}
We already saw in the proof of Proposition~\ref{transimp} that if $|S| = n/2+1$, $n \equiv 2 \mod{4}$ and that in this case, $(G,S)$ has to be in Case 2 of the proof. So we may assume that $n\ge 14$, $k=2$ and $m=n/2 \ge 7$.
{\color{black} Looking at this case closer, we see that no matter whether $C$ is empty or not, $R$ is of size at most 1.} There is, up to isomorphism and duality, a unique way to choose $R$ once $C$ and $L$ have been chosen with $|C| = 1$ and  $|L| = n/2 -1$.
Generators for this case are as follows:

\begin{center}
$\rho_0 = (1,n/2+1)(2,n/2+2)\ldots (n/2,n)$;

$\rho_1 = (2,n/2+2)\ldots (n/2,n)$;

$\rho_i = (i-1,i)(n/2+i-1,n/2+i)$ for $2\leq i \leq n/2$.
\end{center}
Moreover the Sch\"afli type is $\{2,4,3, \ldots, 3\}$ and $\Gamma \cong \Cyc_2\wr \Sym_{n/2}$.

Suppose now that $|L|$ is of size $n/2 - 2$ and $C$ is of size 1. Then by~\cite[Theorem 2]{fl}, there is a unique independent generating set satisfying the string condition for the block action.
In order to have $|S| = n/2 + 1$, we then would need $|R| = 2$. For similar reasons as in the proof of Case 2 of Proposition~\ref{transimp}, this cannot occur.
\end{proof}
Observe that the unique example obtained in Proposition~\ref{prop2} is example (b) in Theorem~\ref{maintheorem}, which is also line 3 of Table~\ref{ploys} for degree 6.

\section{Primitive groups}

In this section we assume that $\Gamma$ is a string C-group of rank $d$, which
is isomorphic to a primitive subgroup of $\Sym_n$ other than $\Sym_n$ or $\Alt_n$.

Asymptotically, our bounds for this case are much better than required for
the main theorem. Rather than give an explicit bound for the rank of a polytope
with primitive group, we give a procedure which will give a good bound for
the rank, and deduce that $d<n/2$ in all but a few cases.

The basic tool, which follows from the Classification of Finite Simple Groups,
is the fact that primitive groups other than $\Sym_n$ and $\Alt_n$ have small order.
The best result along these lines is due to Mar\'oti~\cite{maroti}, which
gives three possibilities for such a group.

\begin{thm}[Mar\'oti~\cite{maroti}]\label{marotiThm}
Let $G$ be a primitive group of degree~$n$ which is not $\Sym_n$ or $\Alt_n$. Then
one of the following possibilities occurs:
\begin{enumerate}
\item For some integers $m,k,l$, we have $n={m\choose k}^l$, and $G$ is a
subgroup of $\Sym_m\wr \Sym_l$, where $\Sym_m$ is acting on $k$-subsets of
$\{1,\ldots,m\}$;
\item $G$ is $\Mat_{11}$, $\Mat_{12}$, $\Mat_{23}$ or $\Mat_{24}$ in its natural
$4$-transitive action;
\item $\displaystyle{|G|\le n\cdot\prod_{i=0}^{\lfloor \log_2n\rfloor-1}
(n-2^i)}$.
\end{enumerate}
\end{thm}

\begin{prop}\label{prim}
With the above hypotheses, $d < n/2$ except for the examples appearing in Table~\ref{primPolys}.
\end{prop}

\begin{table}
\begin{center}
\begin{tabular}{|c|c|c|c|c|}
\hline
Degree&Group&Sch\"afli types&Reference\\
\hline
\hline
10&$\Sym_6$&$[3,3,3,3]$&\cite{LV2006}\\
\hline
6&$\Alt_5$&$[3,5]$, $[5,5]$&\cite{LV2006}\\
&$\Sym_5$&$[3,3,3]$, $[4,5]$, $[4,6]$, $[5,6]$, $[6,6]$&\cite{LV2006}\\
\hline
\end{tabular}
\caption{Primitive string C-groups of degree $n$ and rank $\ge n/2$.}\label{primPolys}
\end{center}
\end{table}

\begin{proof}
{\color{black}Observe first that we may assume the diagram of $\Gamma$ to be connected if $\Gamma$ is primitive. Indeed, suppose the diagram is not connected. Then $d\le2\log_2n < n/2$ when $n\ge 60$}.
The strategy for dealing with the cases of Theorem~\ref{marotiThm} is as follows.

In case (a) with $l=1$, the group is a subgroup of $\Sym_m$ or
$\Alt_m$, so its rank is at most $m-1$; and we have $n={m\choose k}$, so $m$
is roughly $((k!)n)^{1/k}$, much smaller than linear in $n$.

In case (a) with $l>1$, we can replace the primitive action of degree
${m\choose k}^l$ with an imprimitive action of degree $ml$, which is much
smaller; by our results for imprimitive groups, the rank is at most
$(ml/2)+1$.

In case (b), the maximum rank is known from explicit computation (see~\cite{LV2006,HH2010}).
{\color{black}It is respectively 0, 4, 0 and 5 for $\Mat_{11}$, $\Mat_{12}$, $\Mat_{23}$ and $\Mat_{24}$.}
In case (c), there are various methods that can be used.
\begin{itemize}
\item A theorem of Conder~\cite{conder} shows that a string C-group of
rank $d$ has order at least $2^{2d-1}$ when $d\ge 9$. On the other hand, Mar\'oti's
bound in this case is at most $n^{\lfloor \log_2n\rfloor+1}$. So we have
$d\le(\log_2n(\log_2n+1)+1)/2$. A simple calculation shows that $d<n/2$
for $n\ge23$.
\item By the Intersection Property, the boolean lattice of rank $d$ is
embedded into the subgroup lattice of a string C-group of rank $d$. In
particular, such a group has a chain of subgroups of length $d$. The
length of the longest chain of subgroups in a group has been investigated
by Solomon and Turull with various co-authors in a sequence of papers
(see, for example, \cite{st}), and their results can be used to bound $d$
for specific groups.
\end{itemize}

In particular, just by using the bounds and the length of the boolean lattice, we find that a primitive group other than $\Sym_n$ or $\Alt_n$
has rank necessarily smaller than $n/2$, except maybe for the groups listed in Table~\ref{exceptions}. Degree and Number correspond to the numbering in Sims' list of primitive groups of degree $\le 50$ (see for instance~\cite{BL96}). These groups can be constructed in {\sc Magma} using the function {\sf PrimitiveGroupSims}.
\begin{table}
\begin{center}
\begin{tabular}{|c|c|c|c|c|}
\hline
Degree&Number&Group&Maximal rank&Reference\\
\hline
\hline
22&1&${\Mat_{22}}$&0&\cite{LV2006}\\
&2&$\Aut({\Mat_{22}})$&4&\cite{LV2006}\\
\hline
16&19&$2^4:\Alt_7$&At most 7&\\
&20&$2^4:\Alt_8$&At most 7&\\
\hline
15&4&$\PSL_4(2) \cong \Alt_8$&0&\cite{LV2006}\\
\hline
13&7&$\PSL_3(3)$&0&\cite{LV2006}\\
\hline
12&3&${\Mat_{11}}$&0&\cite{LV2006}\\
&4&${\Mat_{12}}$&4&\cite{LV2006}\\
\hline
11&5&$\PSL_2(11)$&4&\cite{LV2006}\\
&6&${\Mat_{11}}$&0&\cite{LV2006}\\
\hline
10&3&$\PSL_2(9) \cong \Alt_6$&0&\cite{LV2006}\\
&4&$\Sym_6$&5&\cite{LV2006}\\
&5&$\PGL_2(9)$&3&\cite{LV2006}\\
&6&$\Mat_{10}$&0&\cite{LV2006}\\
&7&$\PGaL_2(9)$&3&\cite{LV2006}\\
\hline
9&5&$\AGaL(1,9)$&0&\\
&6&$3^2:(2\cdot\Alt_4)$&0&\\
&7&$\AGL(2,3)$&0&\\
&8&$\PSL_2(8)$&0&\cite{LV2006}\\
&9&$\PGaL(2,8)$&3&\cite{LV2006}\\
\hline
8&2&$\AGaL(1,8)$&0&\\
&3&$\PSL_2(7)$&0&\cite{LV2006}\\
&4&$\PGL(2,7)$&3&\cite{LV2006}\\
&5&$\ASL_3(2)$&0&\\
\hline
7&4&$\AGL(1,7)$&0&\\
&5&$\PSL_2(7)$&0&\cite{LV2006}\\
\hline
6&1&$\Alt_5$&3&\cite{LV2006}\\
&2&$\Sym_5$&4&\cite{LV2006}\\
\hline
\end{tabular}
\caption{Groups not excluded by bounds and lengths of lattices.}\label{exceptions}
\end{center}
\end{table}
For most of these groups, the highest rank is known as they were fully investigated in~\cite{LV2006}.

Let $G$ be the group of degree 16, number 19. It cannot be a string C-group of rank 8. Indeed, it has two conjugacy classes of involutions. The smallest one is of size 15 and hence the centraliser of an involution of that class is of order 2688. Suppose $G$ is a string C-group of rank 8. Then, the centraliser of $\rho_0$ must be of order at least $2\times 1728$, a contradiction.

Let $G$ be the primitive group of degree 16, number 20.
It has 4 conjugacy classes of subgroups. Only one of them is of length small enough to have $|C_G(\rho)| \ge 2\times 1728$ for $\rho$ in that conjugacy class. Using {\sc Magma}, it is easy to check that no subgroup of order at least 1728 of that centraliser does not contain $\rho$.
Hence the rank of $G$ cannot be more than 7.

The groups of degree $d\le9$ that were not listed in~\cite{LV2006} are easily dealt with using {\sc Magma}.

So the proposition is proved.
\end{proof}
\section{Consequences}

\begin{coro}
Suppose $G = \Alt_n$ of degree $n$. Let $(G,S)$ be a string C-group with $S = \{\rho_0,\ldots ,\rho_d\}$. If $G_i := \langle \rho_j : j \in \{0,\ldots, d\}\setminus \{i\}\} \rangle$ is transitive imprimitive, then $d \le n/2+1$. 
\end{coro}
\begin{proof}
Suppose first that $n\leq 9$.
In that case, it is easy to check with {\sc Magma} that, unless $n=8$, the rank of $G_i$ is at most $n/2$. In the case where $n=8$, $\Alt_8$ is known not to be a string C-group (see for instance~\cite{LV2006}).
Suppose then that $n>9$. 
By Proposition~\ref{transimp}, we know that $d \le n/2 + 1 + 1$. If $d = n/2 + 2$, then $G_i$ must be as described in Proposition~\ref{prop2}. But this implies that one of the generators of $G_i$ is an odd permutation and $G$ cannot be $\Alt_n$, a contradiction.
Hence $d < n/2 + 2$. 
\end{proof}
\begin{coro}
Suppose $G = \Alt_n$ of degree $n$. Let $(G,S)$ be a string C-group with $S = \{\rho_0,\ldots ,\rho_d\}$. If $G_i := \langle \rho_j : j \in \{0,\ldots, d\}\setminus \{i\}\} \rangle$ is primitive, then $d \le n/2$. 
\end{coro}
\begin{proof}
By Proposition~\ref{prim}, we know that the only primitive groups (distinct from $\Alt_n$ or $\Sym_n$) that are string C-groups of rank at least $n/2$ are those listed in Table~\ref{primPolys}. Since $G_i$ must be a subgroup of $\Alt_n$, it must consist only of even permutations, which is not the case of $\Sym_6$ or $\Sym_5$. Then $G_i = \Alt_5$ and $n=6$. But $\Alt_6$ is not a string C-group (see~\cite{LV2006}). Hence $G_i$ is of rank at most $n/2-1$ and $d\le n/2$.
\end{proof}

\begin{coro}
Suppose $G = \Alt_n$ of degree $n$. Let $(G,S)$ be a string C-group with $S = \{\rho_0,\ldots ,\rho_d\}$. If the rank of $(G,S)$ is at least $n/2+2$, all subgroups $G_i$ must be intransitive.
\end{coro}
\begin{proof}
This is an immediate consequence of the previous two corollaries.
\end{proof}

\section{Acknowledgements}
The first author thanks the Hood Foundation of the University of Auckland for a fellowship which permitted him to spend two fruitful months in Auckland.
This research was also supported by a Marsden grant (12-UOA-083) of the Royal Society of New Zealand, and by FEDER funds through COMPETE- ``Operational Programme Factors of Competitiveness'' (Programa Operacional Fatores de Competitividade) and by Portuguese funds through the Center for Research and Development in Mathematics and Applications (University of Aveiro) and the Portuguese Foundation for Science and Technology (``Fundação para a Ci\^encia e a Tecnologia''), within project PEst-C/MAT/UI4106/2011 with COMPETE number FCOMP-01-0124-FEDER-022690.

\end{document}